\documentclass{amsart}
\usepackage{mathrsfs}

\newtheorem{theorem}{Theorem}[section]
\newtheorem{lemma}[theorem]{Lemma}

\theoremstyle{definition}

\newtheorem{corollary}[theorem]{Corollary}

\theoremstyle{remark}
\newtheorem{remark}[theorem]{Remark}

\numberwithin{equation}{section}

\begin{document}

\title{Derivative formulas for Bessel, Struve and Anger--Weber functions}

\author{Robert E. Gaunt}
\address{School of Mathematics, The University of Manchester, Manchester M13 9PL, UK}
\email{robert.gaunt@manchester.ac.uk}
\thanks{The author is supported by a Dame Kathleen Ollerenshaw Research Fellowship, and also acknowledges support form an EPSRC DPhil Studentship.}

\subjclass[2000]{Primary 26A24, 33C10; Secondary 33C99}

\date{May 2017}

\keywords{Differentiation; Bessel functions; Struve functions; Anger--Weber functions}

\begin{abstract}
We derive formulas for the derivatives of general order for the functions $z^{-\nu}h_{\nu}(z)$ and $z^{\nu}h_{\nu}(z)$, where $h_{\nu}(z)$ is a Bessel, Struve or Anger--Weber function.  
\end{abstract}

\maketitle

\section{Introduction and preliminary results}

Formulas for the derivatives of general order for the functions $z^{-\nu}h_{\nu}(z)$ and $z^{\nu}h_{\nu}(z)$, where $z$ and $\nu$ are complex numbers and $h_{\nu}(z)$ is a Bessel, Struve or Anger--Weber function are established.  In particular, the functions, $h_{\nu}(z)$, that we obtain formulas for are the Bessel functions $J_{\nu}(z)$, $Y_{\nu}(z)$, $I_{\nu}(z)$ and $K_{\nu}(z)$; the Hankel functions $H^{(1)}_{\nu}(z)$ and $H^{(2)}_{\nu}(z)$; the Struve functions $\mathbf{H}_{\nu}(z)$ and $\mathbf{L}_{\nu}(z)$; the Anger--Weber functions $\mathbf{J}_{\nu}(z)$ and $\mathbf{E}_{\nu}(z)$.  For definitions and properties of these functions see, for example,  \cite{NIST, watson}.  


The pair of simultaneous equations
\begin{align} \label{oneone} F_{\nu-1}(z)+F_{\nu+1}(z)&=2F'_{\nu}(z)+f_{\nu}(z), \\
\label{twotwo} F_{\nu-1}(z)-F_{\nu+1}(z)&=\frac{2\nu}{z}F_{\nu}(z)+g_{\nu}(z),
\end{align}
where $f_{\nu}(z)$ and $g_{\nu}(z)$ are arbitrary functions of the complex numbers $\nu$ and $z$, form a generalisation of the recurrence identities that are satisfied by the modified Bessel functions $I_{\nu}(z)$ and $K_{\nu}(z)$, and the modified Struve function $\mathbf{L}_{\nu}(z)$.  These identities can  be found in \cite{NIST, watson}.
Also, the pair of simultaneous equations
\begin{align} \label{threethree} G_{\nu-1}(z)-G_{\nu+1}(z)&=2G'_{\nu}(z)+f_{\nu}(z), \\
\label{fourfour} G_{\nu-1}(z)+G_{\nu+1}(z)&=\frac{2\nu}{z}G_{\nu}(z)+g_{\nu}(z),
\end{align}
where $f_{\nu}(z)$ and $g_{\nu}(z)$ are, again, arbitrary functions of $\nu$ and $z$, form a generalisation of the recurrence identities that are satisfied by the Bessel functions $J_{\nu}(z)$ and $Y_{\nu}(z)$, the Hankel functions $H^{(1)}_{\nu}(z)$ and $H^{(2)}_{\nu}(z)$, the Struve function $\mathbf{H}_{\nu}(z)$, and the Anger--Weber functions $\mathbf{J}_{\nu}(z)$ and $\mathbf{E}_{\nu}(z)$.  Again, these identities can be found in \cite{NIST, watson}.  

The simultaneous equations (\ref{threethree}) and (\ref{fourfour}) were studied by \cite{nielsen}, in which it was shown that the functions $f_{\nu}(z)$ and $g_{\nu}(z)$ must satisfy the relation
\[f_{\nu-1}(z)+f_{\nu+1}(z)-\frac{2\nu}{z}f_{\nu}(z)=g_{\nu-1}(z)-g_{\nu+1}(z)-\frac{2}{z}(zg_{\nu}(z))';\]
and it has been shown by  \cite{watson2} that, if this relation is satisfied, the system can be reduced to a pair of soluble difference equations of the first order.  We may apply similar arguments to the simultaneous equations (\ref{oneone}) and (\ref{twotwo}) to show that the functions $f_{\nu}(z)$ and $g_{\nu}(z)$ must satisfy the relation
\[f_{\nu-1}(z)-f_{\nu+1}(z)-\frac{2\nu}{z}f_{\nu}(z)=g_{\nu-1}(z)+g_{\nu+1}(z)-\frac{2}{z}(zg_{\nu}(z))';\]
and that, if this relation is satisfied, the system can be reduced to a pair of soluble difference equations of the first order.

From equations (\ref{oneone}) and (\ref{twotwo}), one can easily deduce the following formulas:
\begin{align}\label{red}\frac{\mathrm{d}}{\mathrm{d}z}\bigg(\frac{F_{\nu}(z)}{z^{\nu}}\bigg)&=\frac{F_{\nu+1}(z)}{z^{\nu}}+\frac{g_{\nu}(z)-f_{\nu}(z)}{2z^{\nu}}, \\
\label{blue}\frac{\mathrm{d}}{\mathrm{d}z}(z^{\nu}F_{\nu}(z))&=z^{\nu}F_{\nu-1}(z)-\frac{1}{2}z^{\nu}(f_{\nu}(z)+g_{\nu}(z)).
\end{align}
Similarly, from equations (\ref{threethree}) and (\ref{fourfour}), we have:
\begin{align}\label{yellow}\frac{\mathrm{d}}{\mathrm{d}z}\bigg(\frac{G_{\nu}(z)}{z^{\nu}}\bigg)&=-\frac{G_{\nu+1}(z)}{z^{\nu}}+\frac{g_{\nu}(z)-f_{\nu}(z)}{2z^{\nu}}, \\
\label{green}\frac{\mathrm{d}}{\mathrm{d}z}(z^{\nu}G_{\nu}(z))&=z^{\nu}G_{\nu-1}(z)-\frac{1}{2}z^{\nu}(f_{\nu}(z)+g_{\nu}(z)).
\end{align}

Again, formulas (\ref{red})--(\ref{green}) form a generalisation of the well--known formulas (see \cite{NIST, watson}) for the first-order derivatives of Bessel, Struve and Anger--Weber functions.  However, as far as this author is aware, there do not exist simple formulas in the literature for the $n$-th order derivatives of the functions $z^{\pm\nu}F_{\nu}(z)$ and $z^{\pm\nu}G_{\nu}(z)$.  This gap in the literature is filled in by this paper, and we apply these general formulas to obtain formulas for the $n$-th order derivatives of $z^{\pm\nu}h_{\nu}(z)$, where $h_{\nu}(z)$ is a Bessel, Struve or Anger--Weber function.  It should be noted that the proofs of our results rely heavily on the fact that the Bessel, Struve and Anger-Weber functions satisfy the simultaneous equations (\ref{oneone}), (\ref{twotwo}) and (\ref{threethree}), (\ref{fourfour}), and therefore analogous results for the generalized hypergeometric function and other general special functions are beyond the scope of this paper.  This investigation was motivated by occurrence of the $n$-th order derivatives of $z^{-\nu}I_{\nu}(z)$ and $z^{-\nu}K_{\nu}(z)$ in the study of Stein's method for variance-gamma approximation (see \cite{me, gaunt vg}), for which  the simple formulas obtained in this paper proved to be particularly useful. 

\section{Ancillary results} 

Before stating our main results, we establish a result for the coefficients that are present in the formulas.  The coefficients $A_{k}^{n}(\nu)$ and $B_{k}^{n}(\nu)$ are defined, for $n\in\mathbb{N}=\{0,1,2\ldots\}$, $k=0,1,\ldots,n$, and all complex numbers $\nu$, expect the integers $-(k+1),-(k+2),\ldots,-(2k-1),-(2k+1),-(2k+2),\ldots,-(k+n-1)$, and $-(k+2),-(k+3),\ldots,-2k,-(2k+2),-(2k+3),\ldots,-(k+n)$, respectively, as follows:
\begin{align} \label{mcd} A_{k}^{n}(\nu)&=\frac{(2n)!(\nu+2k)\prod_{j=0}^{k-1}(2\nu+2j+1)}{2^{2n-k}(2k)!(n-k)!\prod_{j=0}^n(\nu+k+j)},   \\
\label{gbk} B_{k}^{n}(\nu)&=\frac{(2n+1)!(\nu+2k+1)\prod_{j=0}^{k-1}(2\nu+2j+1)}{2^{2n-k}(2k+1)!(n-k)!\prod_{j=0}^n(\nu+k+j+1)}, 
\end{align}
where we set $\prod_{j=0}^{-1}(2\nu+2j+1)=1$.

\begin{remark}The coefficients $A_{k}^{n}(\nu)$ and $B_{k}^{n}(\nu)$ are equal to zero if and only if $k\geq 1$ and $\nu=-\frac{1}{2}-l$, where $l=0,1,\ldots,k-1$.
\end{remark}

\begin{remark} Let $(x)_n$ denote the Pochhammer symbol $(x)_n=x(x+1)\cdots(x+n-1)$.  Then straightforward calculations yield the following alternative expressions for the coefficients $A_{k}^{n}(\nu)$ and $B_{k}^{n}(\nu)$:
\begin{align*}A_k^n(\nu)&= \frac{(\tfrac{1}{2})_n}{(\nu+1)_n}\frac{(-1)^k(-n)_k(\nu)_k(\nu+\tfrac{1}{2})_k(\tfrac{\nu+2}{2})_k}{(\tfrac{1}{2})_k(\tfrac{\nu}{2})_k(\nu+n+1)_k k!},\\
B_k^n(\nu)&= \frac{(\tfrac{3}{2})_n}{(\nu+2)_n}\frac{(-1)^k(-n)_k(\nu+\tfrac{1}{2})_k(\nu+1)_k(\tfrac{\nu+3}{2})_k}{(\tfrac{3}{2})_k(\tfrac{\nu+1}{2})_k(\nu+n+2)_k k!}.
\end{align*}
We shall not make further use of these expressions, but we do note that the Pochhammer symbol is one of the main tools for complex analysis of generalized hypergeometric functions, such as the Bessel functions, meaning that these representations may be useful in future applications of the differentiation formulas that are derived in this paper.
\end{remark}

\begin{lemma} \label{jess} Let $n\in\mathbb{N}$. Then $A_{k}^{n}(\nu)$ and $B_{k}^{n}(\nu)$ are related as follows
\begin{align}\label{kfc}   B_{k}^{n}(\nu)&=\frac{\nu+k}{\nu+2k}A_{k}^{n}(\nu)+\frac{k+1}{\nu+2k+2}A_{k+1}^{n}(\nu), \quad  0\leq k\leq n-1, \\
\label{tan}   B_{n}^{n}(\nu)&=\frac{\nu+n}{\nu+2n}A_{n}^{n}(\nu), \\ 
\label{sin}   A_0^{n+1}(\nu)&=\frac{1}{2(\nu+1)}B_{0}^{n}(\nu), \\
\label{cos}  A_{k+1}^{n+1}(\nu)&=\frac{2\nu+2k+1}{2(\nu+2k+1)}B_{k}^{n}(\nu)+\frac{2k+3}{2(\nu+2k+3)}B_{k+1}^{n}(\nu),\quad  0\leq k\leq n-1, \\
\label{duff}   A_{n+1}^{n+1}(\nu)&=\frac{2\nu+2n+1}{2(\nu+2n+1)}B_{n}^{n}(\nu), 
\end{align}
and satisfy
\begin{equation}\label{bki} \sum_{k=0}^nA_{k}^{n}(\nu)=\sum_{k=0}^nB_{k}^{n}(\nu)=1. 
\end{equation}
\end{lemma}

\begin{proof}Identities (\ref{kfc})--(\ref{duff}) can be verified by simply substituting the definitions of $A_{k}^{n}(\nu)$ and $B_{k+1}^{n+1}(\nu)$, as given by (\ref{mcd}) and (\ref{gbk}), into both sides of the identities.  

We now prove identity (\ref{bki}).  From (\ref{sin}), (\ref{cos}) and (\ref{duff}), we have that  
\begin{align*}\sum_{k=0}^{n+1}A_{k}^{n+1}(\nu)&=\frac{1}{2(\nu+1)}B_0^{n}(\nu)+\sum_{k=0}^{n-1}\bigg\{\frac{2\nu+2k+1}{2(\nu+2k+1)}B_{k}^{n}(\nu) \\
&\quad+\frac{2k+3}{2(\nu+2k+3)}B_{k+1}^{n}(\nu)\bigg\}+\frac{2\nu+2n+1}{2(\nu+2n+1)}B_{n}^{n}(\nu) \\
&=\sum_{k=0}^n\bigg\{\frac{2\nu+2k+1}{2(\nu+2k+1)}+\frac{2k+1}{2(\nu+2k+1)}\bigg\}B_{k}^{n}(\nu)=\sum_{k=0}^nB_{k}^{n}(\nu).
\end{align*}
A similar calculation shows that
$\sum_{k=0}^nB_{k}^{n}(\nu)=\sum_{k=0}^nA_{k}^{n}(\nu).$
Since $A_0^0(\nu)=B_0^0(\nu)=1$, the result follows.
\end{proof}

\section{Main results}

We are now able to prove our main results.  To simplify the formulas, we define the functions $p_{\nu,l}(z)$ and $q_{\nu,l}(z)$, for $l\in\mathbb{N}$ and $\nu\in\mathbb{C}$, by
\begin{align*}p_{\nu,l}(z)&=\frac{\nu}{2(\nu+l)}\frac{g_{\nu+l}(z)}{z^{\nu}}-\frac{f_{\nu+l}(z)}{2z^{\nu}},  \quad q_{\nu,l}(z)=-\frac{\nu}{2(\nu-l)}z^{\nu}g_{\nu-l}(z)-\frac{1}{2}z^{\nu}f_{\nu-l}(z).
\end{align*}
We use the convention $\frac{0}{0}:=1$, so that $p_{0,0}(z)=\frac{1}{2}(g_0(z)-f_0(z))$ and $q_{0,0}(z)=-\frac{1}{2}(f_0(z)+g_0(z))$.  Also, for $N\geq 1$, we write $[N]$ for the set $\{1,2,\ldots,N\}$, and write $-[N]$ for the set $\{-1,-2,\ldots,-N\}$.  We take $[0]$, $[-1]$, $-[0]$ and $-[-1]$ to be the empty set.  Finally, we let $h^{(n)}(z)$ denote the $n$-th derivative of $h(z)$.  
\begin{theorem} \label{main} Suppose that $F_{\nu}(z)$ satisfies the simultaneous equations (\ref{oneone}) and (\ref{twotwo}), and that $G_{\nu}(z)$ satisfies the simultaneous equations (\ref{threethree}) and (\ref{fourfour}).  Also, suppose that $p_{\nu,l}(z),$ $q_{\nu,l}(z)\in C^{2n}(\mathbb{C})$, for all $l\in\{0,1,\ldots,2n\}$.  Then for $n\in\mathbb{N}$,  
\begin{align} \label{wreck} \frac{\mathrm{d}^{2n}}{\mathrm{d}z^{2n}}\left(\frac{F_{\nu}(z)}{z^{\nu}}\right) &= \sum_{k=0}^n A_{k}^{n}(\nu) \frac{F_{\nu +2k}(z)}{z^{\nu}} +\sum_{j=0}^{n-1}\sum_{k=0}^jA_k^j(\nu)p_{\nu,2k}^{(2n-2j-1)}(z) \nonumber \\
&\quad+\sum_{j=0}^{n-1}\sum_{k=0}^jB_k^j(\nu)p_{\nu,2k+1}^{(2n-2j-2)}(z),  \qquad \nu\in\mathbb{C} \setminus(-[2n-1]), \\
 \label{mend} \frac{\mathrm{d}^{2n+1}}{\mathrm{d}z^{2n+1}}\left(\frac{F_{\nu}(z)}{z^{\nu}}\right) &=  \sum_{k=0}^n B_{k}^{n}(\nu) \frac{F_{\nu +2k+1}(z)}{z^{\nu}} +\sum_{j=0}^{n}\sum_{k=0}^jA_k^j(\nu)p_{\nu,2k}^{(2n-2j)}(z) \nonumber \\
&\quad+\sum_{j=0}^{n-1}\sum_{k=0}^jB_k^j(\nu)p_{\nu,2k+1}^{(2n-2j-1)}(z), \qquad \nu\in\mathbb{C} \setminus(-[2n]), \\
 \label{wreck1} \frac{\mathrm{d}^{2n}}{\mathrm{d}z^{2n}}(z^{\nu}F_{\nu}(z)) &= \sum_{k=0}^n A_{k}^{n}(-\nu) z^{\nu}F_{\nu -2k}(z)+\sum_{j=0}^{n-1}\sum_{k=0}^jA_k^j(-\nu)q_{\nu,2k}^{(2n-2j-1)}(z) \nonumber \\
&\quad+\sum_{j=0}^{n-1}\sum_{k=0}^jB_k^j(-\nu)q_{\nu,2k+1}^{(2n-2j-2)}(z), \qquad \nu\in\mathbb{C} \setminus[2n-1], 
\end{align}
\begin{align}\label{mend1} \frac{\mathrm{d}^{2n+1}}{\mathrm{d}z^{2n+1}}(z^{\nu}F_{\nu}(z)) &=  \sum_{k=0}^n B_{k}^{n}(-\nu) z^{\nu}F_{\nu -2k-1}(z)+\sum_{j=0}^{n}\sum_{k=0}^jA_k^j(-\nu)q_{\nu,2k}^{(2n-2j)}(z) \nonumber \\
&\quad+\sum_{j=0}^{n-1}\sum_{k=0}^jB_k^j(-\nu)q_{\nu,2k+1}^{(2n-2j-1)}(z), \qquad \nu\in\mathbb{C} \setminus[2n],     \\
\label{wreck2} \frac{\mathrm{d}^{2n}}{\mathrm{d}z^{2n}}\left(\frac{G_{\nu}(z)}{z^{\nu}}\right) &= \sum_{k=0}^n (-1)^{n+k} A_{k}^{n}(\nu) \frac{G_{\nu +2k}(z)}{z^{\nu}} \nonumber \\
 &\quad+\sum_{j=0}^{n-1}\sum_{k=0}^j(-1)^{k+j}A_k^j(\nu)p_{\nu,2k}^{(2n-2j-1)}(z) \nonumber \\
&\quad+\sum_{j=0}^{n-1}\sum_{k=0}^j(-1)^{k+j+1}B_k^j(\nu)p_{\nu,2k+1}^{(2n-2j-2)}(z), \: \nu\in\mathbb{C} \setminus(-[2n-1]), \\
\label{mend2} \frac{\mathrm{d}^{2n+1}}{\mathrm{d}z^{2n+1}}\left(\frac{G_{\nu}(z)}{z^{\nu}}\right) &=  \sum_{k=0}^n (-1)^{n+k+1}B_{k}^{n}(\nu) \frac{G_{\nu +2k+1}(z)}{z^{\nu}} \nonumber \\
 &\quad+\sum_{j=0}^{n}\sum_{k=0}^j(-1)^{k+j}A_k^j(\nu)p_{\nu,2k}^{(2n-2j)}(z) \nonumber \\
&\quad+\sum_{j=0}^{n-1}\sum_{k=0}^j(-1)^{k+j+1}B_k^j(\nu)p_{\nu,2k+1}^{(2n-2j-1)}(z), \quad \nu\in\mathbb{C} \setminus(-[2n]), \\
\label{wreck3} \frac{\mathrm{d}^{2n}}{\mathrm{d}z^{2n}}(z^{\nu}G_{\nu}(z)) &= \sum_{k=0}^n (-1)^{n+k}A_{k}^{n}(-\nu) z^{\nu}G_{\nu -2k}(z) \nonumber \\
 &\quad+\sum_{j=0}^{n-1}\sum_{k=0}^j(-1)^{k+j}A_k^j(-\nu)q_{\nu,2k}^{(2n-2j-1)}(z) \nonumber \\
&\quad+\sum_{j=0}^{n-1}\sum_{k=0}^j(-1)^{k+j}B_k^j(-\nu)q_{\nu,2k+1}^{(2n-2j-2)}(z), \quad \nu\in\mathbb{C} \setminus[2n-1], \\
\label{mend3} \frac{\mathrm{d}^{2n+1}}{\mathrm{d}z^{2n+1}}(z^{\nu}G_{\nu}(z)) &=  \sum_{k=0}^n (-1)^{n+k}B_{k}^{n}(-\nu) z^{\nu}G_{\nu -2k-1}(z) \nonumber \\
 &\quad +\sum_{j=0}^{n}\sum_{k=0}^j(-1)^{k+j}A_k^j(-\nu)q_{\nu,2k}^{(2n-2j)}(z) \nonumber \\
&\quad+\sum_{j=0}^{n-1}\sum_{k=0}^j(-1)^{k+j}B_k^j(-\nu)q_{\nu,2k+1}^{(2n-2j-1)}(z), \quad \nu\in\mathbb{C} \setminus[2n],
\end{align}
where we use the convention that $\sum_{k=0}^{-1}a_k=0$.
\end{theorem}

\begin{proof}
We begin by proving formulas (\ref{wreck}) and (\ref{mend}) and do so by induction on $n$.  It is certainly true that (\ref{wreck}) holds for $n=0$ and (\ref{mend}) holds for $n=0$ by (\ref{red}).  Suppose now that (\ref{mend}) holds for $n=m$, where $m\geq 0$.  We therefore have
\begin{align}
\label{bear}\frac{\mathrm{d}^{2m+1}}{\mathrm{d}z^{2m+1}}\left(\frac{F_{\nu}(z)}{z^{\nu}}\right) &=  \sum_{k=0}^m B_{k}^{m}(\nu) \frac{F_{\nu +2k+1}(z)}{z^{\nu}} +\sum_{j=0}^{m}\sum_{k=0}^jA_k^j(\nu)p_{\nu,2k}^{(2m-2j)}(z) \nonumber \\
&\quad+\sum_{j=0}^{m-1}\sum_{k=0}^jB_k^j(\nu)p_{\nu,2k+1}^{(2m-2j-1)}(z), \qquad \nu\in\mathbb{C} \setminus(-[2m]).
\end{align}
Our inductive argument will involve differentiating both sides of (\ref{bear}) and to do so will shall need a formula for the first derivative of the function $z^{-\nu}F_{\nu+\alpha}(z)$, where $\alpha\in\mathbb{N}$.  Applying (\ref{red}) and (\ref{twotwo}) we have, for all $\nu\not= -\alpha$,  
\begin{align}\label{grand}\frac{\mathrm{d}}{\mathrm{d}z}\bigg(\frac{F_{\nu+\alpha}(z)}{z^{\nu}}\bigg)&=\frac{\mathrm{d}}{\mathrm{d}z}\bigg(z^{\alpha}\cdot\frac{F_{\nu+\alpha}(z)}{z^{\nu+\alpha}}\bigg) \nonumber \\
&=\frac{F_{\nu+\alpha+1}(z)}{z^{\nu}}+\frac{\alpha F_{\nu+\alpha}(z)}{z^{\nu+1}}+\frac{g_{\nu+\alpha}(z)-f_{\nu+\alpha}(z)}{2z^{\nu}} \nonumber \\
&=\frac{F_{\nu+\alpha+1}(z)}{z^{\nu}}+\frac{\alpha}{z^{\nu+1}}\cdot\frac{z}{2(\nu+\alpha)}(F_{\nu+\alpha-1}(z)-F_{\nu+\alpha+1}(z) \nonumber \\
&\quad-g_{\nu+\alpha}(z)) +\frac{g_{\nu+\alpha}(z)-f_{\nu+\alpha}(z)}{2z^{\nu}} \nonumber \\
&=\frac{2\nu+\alpha}{2(\nu+\alpha)}\frac{F_{\nu+\alpha+1}(z)}{z^{\nu}}+\frac{\alpha}{2(\nu+\alpha)}\frac{F_{\nu+\alpha-1}(z)}{z^{\nu}} \nonumber \\
&\quad+\frac{\nu}{2(\nu+\alpha)}\frac{g_{\nu+\alpha}(z)}{z^{\nu}}-\frac{f_{\nu+\alpha}(z)}{2z^{\nu}} \nonumber \\
&=\frac{2\nu+\alpha}{2(\nu+\alpha)}\frac{F_{\nu+\alpha+1}(z)}{z^{\nu}}+\frac{\alpha}{2(\nu+\alpha)}\frac{F_{\nu+\alpha-1}(z)}{z^{\nu}}+p_{\nu,\alpha}(z).
\end{align} 
With this formula we may differentiate both sides of (\ref{bear}) to obtain
\begin{align*}
\frac{\mathrm{d}^{2m+2}}{\mathrm{d}z^{2m+2}}\left(\frac{F_{\nu}(z)}{z^{\nu}}\right) &=  \sum_{k=0}^m B_{k}^{m}(\nu) \bigg(\frac{2\nu+2k+1}{2(\nu+2k+1)}\frac{F_{\nu+2k+2}(z)}{z^{\nu}} \\
&\quad+\frac{2k+1}{2(\nu+2k+1)}\frac{F_{\nu+2k}(z)}{z^{\nu}} +p_{\nu,2k+1}(z)\bigg) \\ 
&\quad+\sum_{j=0}^{m}\sum_{k=0}^jA_k^j(\nu)p_{\nu,2k}^{(2m-2j+1)}(z) +\sum_{j=0}^{m-1}\sum_{k=0}^jB_k^j(\nu)p_{\nu,2k+1}^{(2m-2j)}(z) \\
&=  \sum_{k=0}^{m+1} \tilde{A}_{k}^{m}(\nu)\frac{F_{\nu+2k}(z)}{z^{\nu}}  +\sum_{j=0}^{m}\sum_{k=0}^jA_k^j(\nu)p_{\nu,2k}^{(2m-2j+1)}(z) \\ 
&\quad+\sum_{j=0}^{m}\sum_{k=0}^jB_k^j(\nu)p_{\nu,2k+1}^{(2m-2j)}(z), \qquad \nu\in\mathbb{C}\setminus(-[2m+1]),
\end{align*}
where
\begin{align*}\tilde{A}_0^{m+1}(\nu)&=\frac{1}{2(\nu+1)}B_{0}^{m}(\nu), \quad \tilde{A}_{m+1}^{m+1}(\nu)=\frac{2\nu+2m+1}{2(\nu+2m+1)}B_{m}^{m}(\nu) \\
\tilde{A}_{k+1}^{m+1}(\nu)&=\frac{2\nu+2k+1}{2(\nu+2k+1)}B_{k}^{m}(\nu)+\frac{2k+3}{2(\nu+2k+3)}B_{k+1}^{m}(\nu), \quad k=0,1,\ldots ,m-1.
\end{align*}
We see from Lemma \ref{jess} that $\tilde{A}_{k}^{m+1}(\nu)=A_{k}^{m+1}(\nu)$, for all $k=0,1,\ldots ,m+1$.  It therefore follows that if (\ref{mend}) holds for $n=m$ then (\ref{wreck}) holds for $n=m+1$.

We now suppose that (\ref{wreck}) holds for $n=m$, where $m\geq 1$.  If we can show that it then follows that (\ref{mend}) holds for $n=m$ then the proof will be complete.  Our inductive hypothesis is therefore that
\begin{align*} \frac{\mathrm{d}^{2m}}{\mathrm{d}z^{2m}}\left(\frac{F_{\nu}(z)}{z^{\nu}}\right) &= \sum_{k=0}^m A_{k}^{m}(\nu) \frac{F_{\nu +2k}(z)}{z^{\nu}} +\sum_{j=0}^{m-1}\sum_{k=0}^jA_k^j(\nu)p_{\nu,2k}^{(2m-2j-1)}(z)  \\
&\quad+\sum_{j=0}^{m-1}\sum_{k=0}^jB_k^j(\nu)p_{\nu,2k+1}^{(2m-2j-2)}(z),  \qquad \nu\in\mathbb{C} \setminus(-[2m-1]).
\end{align*}
We may use the formula (\ref{grand}) to differentiate the functions $z^{-\nu}F_{\nu+2k}(z)$, for $k\geq 0$, and thus use a similar argument to the first part to obtain
\begin{align*}\frac{\mathrm{d}^{2m+1}}{\mathrm{d}z^{2m+1}}\left(\frac{F_{\nu}(z)}{z^{\nu}}\right) &=  \sum_{k=0}^m \tilde{B}_{k}^{m}(\nu) \frac{F_{\nu +2k+1}(z)}{z^{\nu}} +\sum_{j=0}^{m}\sum_{k=0}^jA_k^j(\nu)p_{\nu,2k}^{(2m-2j)}(z)  \\
&\quad+\sum_{j=0}^{m-1}\sum_{k=0}^jB_k^j(\nu)p_{\nu,2k+1}^{(2m-2j-1)}(z), \qquad \nu\in\mathbb{C} \setminus(-[2m]),
\end{align*}
where
\[\tilde{B}_{m}^{m}(\nu)=\frac{\nu+m}{\nu+2m}A_{m}^{m}(\nu), \quad \tilde{B}_{k}^{m}(\nu)=\frac{\nu+k}{\nu+2k}A_{k}^{m}(\nu)+\frac{k+1}{\nu+2k+2}A_{k+1}^{m}(\nu),\]
and $k=0,1,\ldots,m-1.$
We see from lemma \ref{jess} that $\tilde{B}_{k}^{m}(\nu)=B_{k}^{m}(\nu)$, for all $k=0,1,\ldots ,m$.  It therefore follows that if (\ref{wreck}) holds for $n=m$ then (\ref{mend}) holds for $n=m$, which completes the proof of formulas (\ref{wreck}) and (\ref{mend}). 

We now prove formulas (\ref{wreck1}) and (\ref{mend1}).  Using (\ref{blue}) and (\ref{twotwo}) and a similar calculation to the one used to obtain (\ref{grand}), we have, for all $\nu\not=\alpha$,
\begin{equation}\label{rayol2}\frac{\mathrm{d}}{\mathrm{d}z}(z^{\nu}F_{\nu-\alpha}(z))=\frac{-2\nu+\alpha}{2(-\nu+\alpha)}z^{\nu}F_{\nu-\alpha-1}(z)+\frac{\alpha}{2(-\nu+\alpha)}z^{\nu}F_{\nu-\alpha+1}(z)+q_{\nu,\alpha}(z).
\end{equation}
Comparing  (\ref{rayol2}) with (\ref{grand}), we can see that the formula for $\frac{\mathrm{d}^{2n}}{\mathrm{d}z^{2n}}(z^{\nu}F_{\nu})$ will be similar to formula (\ref{wreck}), with the only difference being that we replace the terms $z^{-\nu}F_{\nu+\alpha}(z)$ by $z^{\nu}F_{\nu-\alpha}(z)$, the terms $A_k^l(\nu)$ and $B_k^l(\nu)$ by $A_k^l(-\nu)$ and $B_k^l(-\nu)$, and the terms $p_{\nu,\alpha}^{(l)}(z)$ by $q_{\nu,\alpha}^{(l)}(z)$.  Formulas (\ref{wreck2}), (\ref{mend2}), (\ref{wreck3}) and (\ref{mend3}) can be verified using similar calculations.
\end{proof}

We now apply Theorem \ref{main} to obtain formulas for the derivatives of any order for the functions $z^{-\nu}h_{\nu}(z)$ and $z^{\nu}h_{\nu}(z)$, where $h_{\nu}(z)$ is a Bessel, Struve, or Anger--Weber function.  The formulas for Bessel functions are particularly simple: 

\begin{corollary} \label{fraction}
Let $\mathscr{C}_{\nu}(z)$ denote $J_{\nu}(z)$, $Y_{\nu}(z)$, $H_{\nu}^{(1)}(z)$, $H_{\nu}^{(2)}(z)$ or any linear combination of these functions, in which the coefficients are independent of $\nu$ and $z$.  Then for $n\in\mathbb{N}$,
\begin{align*}  \frac{\mathrm{d}^{2n}}{\mathrm{d}z^{2n}}\left(\frac{\mathscr{C}_{\nu}(z)}{z^{\nu}}\right) &= \sum_{k=0}^n (-1)^{n+k} A_{k}^{n}(\nu) \frac{\mathscr{C}_{\nu +2k}(z)}{z^{\nu}}, \qquad \nu\in\mathbb{C} \setminus(-[2n-1]), \\
  \frac{\mathrm{d}^{2n+1}}{\mathrm{d}z^{2n+1}}\left(\frac{\mathscr{C}_{\nu}(z)}{z^{\nu}}\right) &=  \sum_{k=0}^n (-1)^{n+k+1}B_{k}^{n}(\nu) \frac{\mathscr{C}_{\nu +2k+1}(z)}{z^{\nu}}, \qquad \nu\in\mathbb{C} \setminus(-[2n]), \\
  \frac{\mathrm{d}^{2n}}{\mathrm{d}z^{2n}}(z^{\nu}\mathscr{C}_{\nu}(z)) &= \sum_{k=0}^n (-1)^{n+k}A_{k}^{n}(-\nu) z^{\nu}\mathscr{C}_{\nu -2k}(z), \qquad \nu\in\mathbb{C} \setminus[2n-1], \\
 \frac{\mathrm{d}^{2n+1}}{\mathrm{d}z^{2n+1}}(z^{\nu}\mathscr{C}_{\nu}(z)) &=  \sum_{k=0}^n (-1)^{n+k}B_{k}^{n}(-\nu) z^{\nu}\mathscr{C}_{\nu -2k-1}(z), \qquad \nu\in\mathbb{C} \setminus[2n]. 
\end{align*}
Now let $\mathscr{L}_{\nu}(z)$ denote $I_{\nu}(z)$, $ e^{\nu \pi i}K_{\nu}(z)$ or any linear combination of these functions, in which the coefficients are independent of $\nu$ and $z$.  Then for $n\in\mathbb{N}$,
\begin{align}  \frac{\mathrm{d}^{2n}}{\mathrm{d}z^{2n}}\left(\frac{\mathscr{L}_{\nu}(z)}{z^{\nu}}\right) &= \sum_{k=0}^n A_{k}^{n}(\nu) \frac{\mathscr{L}_{\nu +2k}(z)}{z^{\nu}}, \qquad \nu\in\mathbb{C} \setminus(-[2n-1]), \nonumber \\
\frac{\mathrm{d}^{2n+1}}{\mathrm{d}z^{2n+1}}\left(\frac{\mathscr{L}_{\nu}(z)}{z^{\nu}}\right) &=  \sum_{k=0}^n B_{k}^{n}(\nu) \frac{\mathscr{L}_{\nu +2k+1}(z)}{z^{\nu}}, \qquad \nu\in\mathbb{C} \setminus(-[2n]), \nonumber \\
 \frac{\mathrm{d}^{2n}}{\mathrm{d}z^{2n}}(z^{\nu}\mathscr{L}_{\nu}(z)) &= \sum_{k=0}^n A_{k}^{n}(-\nu) z^{\nu}\mathscr{L}_{\nu -2k}(z), \qquad \nu\in\mathbb{C} \setminus[2n-1], \nonumber \\
  \frac{\mathrm{d}^{2n+1}}{\mathrm{d}z^{2n+1}}(z^{\nu}\mathscr{L}_{\nu}(z)) &=  \sum_{k=0}^n B_{k}^{n}(-\nu) z^{\nu}\mathscr{L}_{\nu -2k-1}(z), \qquad \nu\in\mathbb{C} \setminus[2n]. \nonumber 
\end{align}
\end{corollary}

\begin{proof}We have that $\mathscr{C}_{\nu}(z)$ satisfies the system of equations (\ref{threethree}) and (\ref{fourfour}), with $f_{\nu}(z)=g_{\nu}(z)=0$ (see \cite{NIST}, Section 10.6).  Hence, the formulas involving $L_{\nu}(z)$ follow immediately from taking $p_{\nu,l}(z)=q_{\nu,l}(z)=0$ in formulas (\ref{wreck2})--(\ref{mend3}) of Theorem \ref{main}.  The formulas involving $\mathscr{L}_{\nu}(z)$ are derived similarly.
\end{proof}

\begin{corollary}The $n$-th derivatives of the functions $z^{-\nu}\mathbf{H}_{\nu}(z)$ and $z^{\nu}\mathbf{H}_{\nu}(z)$ satisfy formulas similar to (\ref{wreck2})--(\ref{mend3}), with the only difference being that the terms $p_{\nu,l}^{(m)}(z)$ and $q_{\nu,l}^{(m)}(z)$ are replaced by $t_{\nu,l}^m(z)$ and $u_{\nu,l}^m(z)$, respectively, functions $t_{\nu,l}^{m}(z)$ and $u_{\nu,l}^{m}(z)$ are defined, for $l,m\in\mathbb{N}$, by
\begin{align*}t_{\nu,l}^{m}(z)&=\frac{(2\nu+l)(\frac{1}{2})^{\nu+l+1}}{\sqrt{\pi}(\nu+l)\Gamma(\nu+l+\frac{3}{2})}(l)_{m}z^{l-m}, \\
u_{\nu,l}^{m}(z)&= \frac{l(\frac{1}{2})^{\nu-l+1}}{\sqrt{\pi}(l-\nu)\Gamma(\nu-l+\frac{3}{2})}(2\nu-l)_{m}z^{2\nu-l-m}, 
\end{align*}
where $(x)_n$ denotes the Pochhammer symbol $(x)_n=x(x+1)\cdots(x+n-1)$. 

Similarly, the $n$-th derivatives of the functions $z^{-\nu}\mathbf{L}_{\nu}(z)$ and $z^{\nu}\mathbf{L}_{\nu}(z)$ satisfy formulas similar to (\ref{wreck})--(\ref{mend1}), with the only difference being that the terms $p_{\nu,l}^{(m)}(z)$ and $q_{\nu,l}^{(m)}(z)$ are replaced by $t_{\nu,l}^m(z)$ and $u_{\nu,l}^m(z)$.

The $n$-th derivatives of the functions $z^{-\nu}\mathbf{J}_{\nu}(z)$ and $z^{\nu}\mathbf{J}_{\nu}(z)$ satisfy formulas similar to (\ref{wreck2})--(\ref{mend3}), with the only difference being that the terms $p_{\nu,l}^{(m)}(z)$ and $q_{\nu,l}^{(m)}(z)$ are replaced by $t_{\nu,l}^m(z)$ and $u_{\nu,l}^m(z)$, respectively, where the functions $v_{\nu,l}^{m}(z)$ and $w_{\nu,l}^{m}(z)$ are defined, for $l,m\in\mathbb{N}$, by
\begin{align*}v_{\nu,l}^{m}(z)&=-\frac{\nu}{\pi(\nu+l)}\sin(\pi(\nu+l))(-\nu-1)_{m}z^{-\nu-m-1},  \\
w_{\nu,l}^{m}(z)&=\frac{\nu}{\pi(\nu-l)}\sin(\pi(\nu-l))(\nu-1)_{m}z^{\nu-m-1}.
\end{align*} 
The formulas for the Weber function $\mathbf{E}_{\nu}(z)$ are similar, with the only difference being that $\sin(\pi(\nu+k))$ is replaced by $1-\cos(\pi(\nu+k))$, for $k\in\mathbb{Z}$.
\end{corollary}

\begin{proof}We first establish the formulas for the $n$-th derivatives involving Struve functions.  The function $\mathbf{H}_{\nu}(z)$ satisfies the system of equations (\ref{threethree}) and (\ref{fourfour}), with $f_{\nu}(z)=-g_{\nu}(z)=-\frac{(\frac{1}{2}z)^{\nu}}{\sqrt{\pi}\Gamma(\nu+\frac{3}{2})}$ (see \cite{NIST}, Section 11.4).  Therefore
\begin{align*}p_{\nu,l}^{(m)}(z)&=\frac{\mathrm{d}^{m}}{\mathrm{d}z^m}\bigg(\frac{\nu}{2(\nu+l)}\frac{g_{\nu+l}(z)}{z^{\nu}}-\frac{f_{\nu+l}(z)}{2z^{\nu}}\bigg) =\frac{(2\nu+l)(\frac{1}{2})^{\nu+l+1}}{\sqrt{\pi}(\nu+l)\Gamma(\nu+l+\frac{3}{2})}\frac{d^m}{dz^m}(z^l) \\
&=\frac{(2\nu+l)(\frac{1}{2})^{\nu+l+1}}{\sqrt{\pi}(\nu+l)\Gamma(\nu+l+\frac{3}{2})}(l)_{m}z^{l-m} =t_{\nu,l}^m(z).
\end{align*}
Similarly, we can show that $q_{\nu,l}^{(m)}(z)=u_{\nu,l}^{m}(z)$.  We then apply formulas (\ref{wreck2})--(\ref{mend3}) of Theorem \ref{main}, and this gives the desired formulas for the derivatives of the functions $z^{-\nu}\mathbf{H}_{\nu}(z)$ and $z^{\nu}\mathbf{H}_{\nu}(z)$.  The formulas involving $\mathbf{L}_{\nu}(z)$ are derived similarly.

The Anger--Weber functions  $\mathbf{J}_{\nu}(z)$ and $\mathbf{E}_{\nu}(z)$ satisfy the system of equations (\ref{threethree}) and (\ref{fourfour}), with $f_{\nu}(z)=0$, $g_{\nu}(z)=-\frac{2}{\pi z}\sin(\pi \nu)$, and $f_{\nu}(z)=0$, $g_{\nu}(z)=-\frac{2}{\pi z}(1-\cos(\pi\nu))$, respectively (see \cite{NIST}, Section 11.10).  We then apply Theorem \ref{main} and some simple calculations, as we did in the proof of formulas for the $n$-th derivatives involving Struve functions, to obtain the desired formulas.
\end{proof}

\bibliographystyle{amsplain}

\begin{thebibliography}{9}

\bibitem{me} Gaunt, R.E.  \emph{Rates of Convergence of Variance-Gamma Approximations via Stein's Method.}  DPhil thesis, University of Oxford, 2013.

\bibitem{gaunt vg} Gaunt, R. E.  Variance-Gamma approximation via Stein's method.  \emph{Electron. J. Probab.} $\mathbf{19}$ No. 38 (2014), pp. 1--33.

\bibitem{nielsen} Nielsen, N.  \emph{Handbuch der Theorie der Cylinderfunktionen.}  Leipzig, 1904.

\bibitem{NIST} Olver, F.W.J., Lozier, D.W., Boisvert, R.F. and Clark, C.W. \emph{NIST Handbook of
Mathematical Functions}, Cambridge University Press, 2010.

\bibitem{watson} Watson, G.N. \emph{A Treatise on the Theory of Bessel Functions}, 2nd ed. Cambridge, England: Cambridge University Press, 1966. 

\bibitem{watson2} Watson, G.N.  On Nielsen's functional equations.  \emph{Messenger}, XLVIII. (1919), pp. 49--53.

\end{thebibliography}

\end{document}